\documentclass{amsart}

\usepackage[latin1]{inputenc} 
\usepackage[T1]{fontenc} 
\usepackage[english, frenchb]{babel} 

\newtheorem{cor}{Corollaire}
\newtheorem{thm}{Théorème}
\newtheorem{lem}{Lemme}
\newtheorem{prop}{Proposition}
\theoremstyle{definition}
\newtheorem{defn}{Définition}

\title{Les groupes de Burger-Mozes ne sont pas kählériens} 
\author{Thibaut Delcroix} 
\date{} 
\address{Institut Fourier, CNRS UMR 5582, Université Grenoble 1, 100 rue des maths, BP 74, 38402 Saint-Martin d'Hères, France}
\email{thibaut.delcroix@ujf-grenoble.fr}

\begin{document}

\maketitle

\selectlanguage{english}
\begin{abstract}
Burger and Mozes constructed examples of infinite simple groups which are lattices in the group of
automorphisms of a cubical building. We show that there can be no morphism with finitely 
generated kernel from a Kähler group to one of these groups. We obtain as a consequence that 
these groups are not Kähler.
\end{abstract}

\selectlanguage{frenchb}
\begin{abstract}
Burger et Mozes ont construit des exemples de groupes simples infinis, qui sont des réseaux
dans le groupe des automorphismes d'un immeuble cubique. On montre qu'il n'existe pas de morphisme
d'un groupe kählérien vers l'un de ces groupes dont le noyau soit finiment engendré. On en 
déduit que ces groupes ne sont pas kählériens. 
\end{abstract}

\section{Introduction}

Bien que des exemples non résiduellement 
finis aient déjà été exhibés, l'existence de groupes fondamentaux de variétés kählériennes compactes, 
appelés grou\-pes kählériens,
infinis et simples n'est pas encore connue. 
Les candidats explicites à ce rôle sont en réalité assez peu nombreux. En effet, un groupe
kählérien est de présentation finie et les exemples de groupes infinis simples de présentation finie
sont des objets dont la construction n'est pas aisée. 

Burger et Mozes ont exhibé dans \cite{BM}  les premiers ex\-emples de 
groupes infinis simples de présentation finie et sans torsion. 
Il existe une famille infinie $\mathcal{F}_{BM}$ de groupes non isomorphes telle que tout groupe 
$\Gamma$ de  $\mathcal{F}_{BM}$
est infini, simple, sans torsion et de présentation finie. 
Ces groupes sont des réseaux du groupe des automorphismes d'un immeuble cubique.
Leur construction se fait en deux étapes. 

La première est de prouver un analogue du théorème du 
sous-groupe normal de Margulis pour les réseaux dans les produits d'arbres, c'est-à-dire montrer 
que pour certains réseaux dans le groupe d'automorphismes d'un produit d'arbres, tout sous-groupe 
distingué est d'indice fini dans le réseau. Ceci est établi par Burger et Mozes dans le cas de 
réseaux dont les projections dans les groupes d'automorphismes de chacun des arbres vérifient des 
propriétés de transitivité forte ($\infty$-transitivité ou transitivité du stabilisateur de n'importe 
quel sommet sur chaque sphère centrée en ce sommet). 

La seconde étape est de construire un réseau vérifiant les hypothèses de ce théorème du sous-groupe normal 
et qui soit de plus non-résiduellement fini. Un tel réseau est alors virtuellement simple, c'est-à-dire 
qu'il admet un sous-groupe simple d'indice fini. 

La motivation de ce travail est de montrer que ces groupes ne sont pas kählériens et même une propriété 
plus forte : il n'existe pas de morphisme ayant un noyau de type fini d'un groupe kählérien vers un de ces 
groupes de Burger-Mozes. 

Nous démontrons donc le théorème suivant :

\begin{thm}
\label{noyau}
Soit $\Gamma$ un groupe de Burger-Mozes, $X$ une variété kählérienne et $G:=\pi_1(X)$ son groupe fondamental.
Supposons qu'il existe un morphisme surjectif $\rho: G\rightarrow \Gamma$, alors son noyau $\mathrm{ker}(\rho)$
n'est pas de type fini.
\end{thm}

En considérant le morphisme identité d'un groupe de Burger-Mozes, on obtient le corollaire suivant :

\begin{cor}
\label{nonK}
Les groupes de Burger-Mozes ne sont pas kählériens.
\end{cor}

Nous montrerons également le corollaire de manière différente en prouvant qu'un éventuel groupe kählérien
infini et simple vérifierait la propriété (FA).

\section{Factorisation des actions sur les arbres des groupes kählériens}

On rappelle qu'une \emph{orbisurface de Riemann hyperbolique} $S$  est un quotient du disque 
unité $D$ de $\mathbb{C}$ par un sous-groupe discret cocompact $P$  de 
$\mathrm{PSL}_2(\mathbb{R})$. On appelle $P$ le \emph{groupe fondamental} (orbifolde) de $S$ et on 
le note $\pi_1(S)$.

Une application $f$ d'une variété complexe $X$ vers $S$
est dite holomorphe si elle se relève en une application holomorphe du revêtement 
fondamental $\widetilde{X}$ de $X$ vers $D$. Une telle application induit un
morphisme de groupes fondamentaux $f_*:\pi_1(X)\rightarrow \pi_1(S)$.

On appellera \emph{fibration} une application holomorphe (orbifolde ou non)
surjective à fibres connexes.
Si $f:X\rightarrow S$ est une fibration sur une surface de Riemann alors 
le morphisme $f_*:\pi_1(X)\rightarrow \pi_1(S)$ est surjectif.
Cette propriété reste vraie pour les fibrations sur des orbisurfaces de Riemann.

L'outil pour démontrer le théorème \ref{nonK} est le résultat suivant.
\begin{thm}
\label{GS}
Soient $X$ une variété kählérienne compacte, $G:=\pi_1(X)$ son groupe fondamental, 
$T$ un arbre infini ayant plus de trois bouts, et
$\rho:G\rightarrow \mathrm{Aut}(T)$
un morphisme tel que l'action associée de $G$ sur $T$ ne stabilise aucun bout, ni aucun sous-arbre
propre de $T$. 

Alors 
il existe 
une orbisurface de Riemann hyperbolique $S_{\mathrm{orb}}$ et
une fibration $f:X\rightarrow S_{\mathrm{orb}}$,
telle que $\rho$ factorise par $f_*$, c'est-à-dire qu'il existe un morphisme 
$\phi : \pi_1(S_{\mathrm{orb}}) \rightarrow \mathrm{Aut}(T)$ tel que 
$$\rho= \phi \circ f_*.$$
\end{thm}

L'existence de la fibration d'un revêtement fini vers une 
surface de Riemann est un résultat classique issu des travaux de Gromov et Schoen \cite{GS} dans le 
cas de l'action d'un  amalgame sur un arbre localement fini, au 
moyen de la théorie des applications harmoniques à valeurs dans des complexes simpliciaux développée dans leur article.
Leurs résultats sont également expliqués dans le livre \cite{ABC}. 
Une autre preuve de leur résultat, valable pour un amalgame quelconque peut être trouvée dans \cite{NR}.

Pour démontrer le théorème \ref{noyau}, l'existence d'une telle application ne suffit pas, 
on a besoin d'un résultat plus précis de factorisation du morphisme $\rho$.
De tels résultats ont été récemment utilisés, pour d'autres groupes, 
par Corlette et Simpson  \cite{CS}, Delzant et Py \cite{DP} ou Py \cite{Py}. Une condition nécessaire 
pour obtenir une 
factorisation est la minimalité de l'action. Ici ceci correspond à ne fixer aucun sous-arbre propre.
On obtiendra la factorisation voulue en remplaçant la surface de Riemann par une 
orbisurface de Riemann.

Bien que des résultats similaires à ce théorème soient nombreux dans la littérature, nous
en donnons la démonstration détaillée.

\begin{proof}
\textbf{Première étape : existence d'une application harmonique équivariante}

La première étape est de construire une application harmonique équivariante non constante $u$ du
revêtement universel de $X$ vers l'arbre $T$. 
Puisqu'on ne requiert pas que l'arbre soit localement fini, 
il faut faire appel aux travaux de Korevaar et Schoen. 

\begin{prop}
\cite[Corollary 2.3.2]{KS}
Si le groupe fondamental d'une variété riemannienne compacte agit sur un 
arbre alors soit l'action fixe un bout, soit il existe une application
harmonique équivariante du revêtement universel de la variété vers l'arbre.
\end{prop}

Les applications harmoniques sont les minimiseurs d'une certaine énergie
(voir \cite{GS} et \cite{KS1}).
Korevaar et Schoen construisent dans \cite{KS1}
une suite minimisante formée d'applications équivariante.
Pour éviter que la limite ne parte à l'infini, il est nécessaire d'imposer la 
condition que $G$ ne fixe pas de bouts de $T$.

L'application harmonique obtenue est ici non constante par équivariance : si elle 
était constante alors son image serait fixée par le groupe $G$, ce qui 
contredit l'hypothèse que $G$ ne fixe pas de sous-arbre propre.

Dans la suite, on utilisera des résultats de \cite{GS} où les résultats des auteurs 
ne s'appliquent 
a priori qu'à des arbres localement finis. Cependant gr\^ace à Sun \cite{sun}, 
l'image d'un voisinage assez petit d'un point par une application harmonique
vers un arbre est contenue dans un sous-arbre localement fini. 
Ainsi, les résultats locaux s'étendent aux arbres non localement finis.
\newline

\textbf{Deuxième étape : donnée spectrale}

Un point $x$ de $\widetilde{X}$ est dit \emph{régulier} si l'image par $u$ d'un voisinage de $x$ 
est contenue dans une géodésique de $T$. Il est dit \emph{singulier} sinon. 
Comme une géodésique de $T$ est une copie de $\mathbb{R}$, l'application $u$ peut alors être considérée
au voisinage d'un point régulier comme une application à valeurs dans $\mathbb{R}$.

\begin{thm}
\cite[Theorem 6.4]{GS}
L'ensemble des points singuliers de $\widetilde{X}$ est de codimension de Hausdorff au moins 2.
\end{thm}

L'application $u$ est en fait \emph{pluriharmonique} sur l'ensemble des points réguliers par 
\cite[Theorem 7.3]{GS}, c'est à dire que, considérée comme une fonction à valeurs réelles au voisinage d'un point
régulier, elle est vraiment une fonction pluriharmonique (i.e. $\partial\overline{\partial}u=0$).
Ainsi, la différentielle de $u$ au voisinage d'un point régulier, que l'on peut définir à valeur 
dans $\mathbb{R}$, au signe près car les géodésiques n'ont pas d'orientation fixée, est la partie 
réelle d'une 1-forme holomorphe $\theta$.

Le carré $\beta$ de $\theta$ est une différentielle quadratique holomorphe bien définie 
sur le lieu des points réguliers.
Le résultat \cite[Theorem 6.4]{GS} cité ci-dessus permet d'étendre $\beta$ en une différentielle quadratique 
définie sur $\widetilde{X}$ tout entier. Puisque $u$ est équivariante, $\beta$ descend sur $X$ 
en une différentielle quadratique holomorphe définie globalement, que l'on note encore $\beta$.

Considérons le revêtement fini (qui peut être ramifié) $Y$ de $X$ associé canoniquement à 
cette différentielle quadratique holomorphe. Il s'agit de la sous-variété de l'espace total 
du fibré cotangent $\Omega_X$ définie par :
$$Y=\{(x,w)\in \Omega_X:w^2=\beta(x)\}.$$
Alors le relèvement de $\beta$ à $Y$ est une 
différentielle quadratique qui est globalement le carré d'une forme différentielle holomorphe $\alpha$. 

Si le revêtement $Y$ est trivial, cela signifie que $\beta$ était déjà globalement un carré sur $X$. 
Dans ce cas on note $Y:=X$ dans la suite, et il ne sera pas nécessaire de faire un quotient lors de 
la quatrième étape.
\newline

\textbf{Troisième étape : Construction de la fibration}

On va maintenant utiliser l'article \cite{Sim}  de Simpson pour obtenir une fibration 
sur une surface de Riemann. Dans cette troisième étape, on n'utilise pas entièrement 
la minimalité de l'action, mais juste le fait que $G$ ne stabilise pas de paire de bouts.

Notons $\mathrm{Alb}(Y)$ la \emph{variété de Albanese} de $Y$, c'est à dire le tore complexe
défini par 
$$\mathrm{Alb}(Y)=H^0(Y,\Omega_Y)^*/H_1(Y,\mathbb{Z}).$$
Soit $\alpha_{\mathrm{Alb}}$ 
une 1-forme sur $\mathrm{Alb}(Y)$ telle que $\alpha$ soit le tiré en arrière 
de $\alpha_{\mathrm{Alb}}$ par l'\emph{application d'Albanese}  
$x\mapsto [\alpha \mapsto \int_{x_0}^x\alpha]$ de $Y$ vers $\mathrm{Alb}(Y)$ 
(pour un point base $x_0\in Y$ fixé).

Soit $A$ la variété abélienne minimale de laquelle $\alpha$ provient, 
définie par $A:=\mathrm{Alb}(Y)/B$ où 
$B$ est la somme des sous-variétés abélienne de $\mathrm{Alb}(Y)$ sur lesquelles 
$\alpha_{\mathrm{Alb}}$ est nulle.
Notons $\alpha_A$ la projection de $\alpha_{\mathrm{Alb}}$ sur $A$.

Enfin, notons $Y'=Y\times_AV$ le produit fibré, de fibre $A$, de $Y$ avec le revêtement
universel $V$ de $A$. 
Soit $z_0\in Y'$ un point qui est projeté sur $x_0$ dans $Y$, et $v_0$ sa projection 
dans $V$.
Le tiré en arrière de $\alpha_A$ à $V$ est exact, et s'écrit donc comme la différentielle
d'une unique fonction $g_V:V\rightarrow \mathbb{C}$ qui s'annule en $v_0$ . On note $g$
la composition de $g_V$ avec la projection $Y'\rightarrow V$.

On s'intéresse maintenant à la dimension de l'image de $Y$ dans $A$ par le quotient de 
l'application d'Albanese, pour pouvoir utiliser le résultat de Simpson.

\begin{lem}
La dimension de l'image de $Y$ dans $A$ est égale à 1.
\end{lem}

\begin{proof}
Si la 1-forme $\alpha$ est identiquement nulle, alors l'application harmonique 
de départ $u$ est constante, mais ce n'est pas le cas par construction.
Donc la dimension de l'image de $Y$ dans $A$ est supérieure ou égale à 1.
Supposons maintenant que la dimension de l'image de $Y$ dans $A$ soit supérieure ou égale à 2.

Par le théorème de Simpson \cite{Sim} énoncé dans le cas projectif, mais qui reste valable 
dans le cas kählérien, pour tout $t\in \mathbb{C}$, la fibre $g^{-1}(t)\subset Y'$ est 
connexe et son groupe fondamental se surjecte sur celui de $Y'$.

Le relèvement à $\widetilde{Y}$ de cette fibre $g^{-1}(t)$ 
est une réunion de composantes connexes que l'on note $\widetilde{Y}_{t,k}$.
La différentielle de $u$ est nulle sur les composantes $\widetilde{Y}_{t,k}$ par 
construction de $g$. 
On choisit une telle composante $\widetilde{Y}_{t,k}$, qui s'envoie donc sur un seul 
point $q\in T$ par $u$. 

Si $x\in \widetilde{Y}_{t,k}$ est un point, et si $\gamma \in \pi_1(g^{-1}(t))$ alors l'image 
de $x$ par $\gamma$ est encore dans la même composante.
En particulier, on a $u(\gamma x)=q=u(x)$.
Or, par équivariance de $u$, on a $u(\gamma x)=\rho(\gamma)u(x)$, donc on obtient
$\rho(\gamma)q=q$.
En d'autres termes, l'action de $\pi_1(g^{-1}(t))$ sur $T$ fixe $q$.
Puisque $\pi_1(g^{-1}(t))$ surjecte sur $\pi_1(Y')$, l'action de $\pi_1(Y')$ sur $T$ fixe $q$ également.
Ceci est vérifié pour tout $q$ dans l'image de $u$, donc l'action de $\pi_1(Y')$ est 
triviale sur l'image de $u$. L'action de $\pi_1(Y')$ est donc triviale sur le sous-arbre
$T'$ de $T$ qui est l'enveloppe convexe de l'image de $u$.

Le revêtement $Y'\rightarrow Y$ est un revêtement galoisien infini de groupe de Galois
le quotient de $\pi_1(Y)$ donné par le morphisme $\pi_1(Y)\rightarrow \pi_1(A)$. 
En particulier le groupe de Galois de $Y'/Y$ est abélien. 
Donc $\pi_1(Y')$ contient le sous-groupe des commutateurs de $\pi_1(Y)$.
Si la dimension de l'image de $Y$ dans $A$ est supérieure ou égale à 2, alors
le sous-groupe des commutateurs agit trivialement sur $T'$.

Ceci est impossible. En effet, les hypothèses sur l'action de $G$ sur $T$ impliquent 
que $G$ ne fixe ni sommet, ni bout, ni paire de bouts de $T$. Par conséquent, 
$G$ ne peut pas agir sur $T'$ en fixant un sommet, un bout ou une paire de bouts.
Il découle alors du théorème principal de \cite{PV} que l'image de $G$ dans $\mathrm{Aut}(T')$
contient un sous-groupe libre non abélien, donc elle ne peut pas être un groupe abélien.
\end{proof}

On rappelle que toute application holomorphe $\phi:X_1\rightarrow X_2$ partant d'une 
variété compacte $X_1$ admet une factorisation $\phi=\phi_a\circ \phi_b$ appelée 
\emph{factorisation de Stein} \cite{Ste} où $\phi_a: X_2'\rightarrow X_2$ est un revêtement fini sur son
image et
$\phi_b: X_1\rightarrow X_2'$ est une fibration. 

La factorisation de Stein de $Y\rightarrow A$ donne, puisque la dimension de l'image de $Y$ dans $A$ est 1,
$$Y \longrightarrow S' \longrightarrow A$$
où $S'$ est une surface de Riemann compacte, et $f' : Y\rightarrow S'$ est une fibration. 
\newline

\textbf{Quatrième étape : factorisation de la représentation}

La partie réelle de $\alpha$ est une 1-forme fermée car c'est localement la différentielle 
de $u$, et définit donc un feuilletage singulier.
Les fibres de $f'$ sont contenues 
dans les feuilles de ce feuilletage, puisque 
$f'$ factorise $Y\rightarrow A$. On en déduit que l'application
$u$ est localement constante sur le relèvement des fibres de $f'$.

L'image du groupe fondamental d'une fibre générique est un groupe $H$ distingué dans 
$\pi_1(Y)$. En effet, notons $S'_{\mathrm{reg}}$ la surface de Riemann $S'$ privée
des points singuliers de $f'$, et $Y_{\mathrm{reg}}$ son image réciproque.
Alors $f'\mid_{S'_{\mathrm{reg}}}$ vérifie la propriété de relèvement des homotopies, 
et on a une suite exacte 
$$1 \longrightarrow \pi_1(F) \longrightarrow \pi_1(Y_{\mathrm{reg}}) \longrightarrow \pi_1(S'_{\mathrm{reg}}) \longrightarrow 1$$
où $\pi_1(F)$ est le groupe fondamental d'une fibre. En particulier, $\pi_1(F)$
est distingué dans $\pi_1(Y_{\mathrm{reg}})$.
Or le morphisme $\pi_1(Y_{\mathrm{reg}}) \rightarrow \pi_1(Y)$ induit par l'inclusion 
est surjectif, donc l'image de $\pi_1(F)$ est distinguée dans $\pi_1(Y)$

D'autre part, l'image $H$ du groupe fondamental d'une fibre générique fixe au moins un 
point de $T$, puisqu'il fixe point par point 
l'image du relèvement de la fibre. Notons $T'$ le sous-arbre de $T$ fixé 
par $H$. Si $g\in \pi_1(Y)$ et $p \in T'$ alors 
$$H\cdot g\cdot p=g\cdot H\cdot p=g\cdot p$$
donc $g\cdot p \in T'$.
On en déduit que $T'$ est stabilisé par $G$ donc $T'$ est en fait l'arbre $T$ 
tout entier par minimalité de l'action.
Par conséquent, l'image du groupe fondamental d'une fibre générique est triviale dans $\rho(G)$. 

En considérant $S'_o$ la base orbifolde \cite[Définition 4.2]{Cam} de la fibration $f'$, 
on a alors une suite exacte 
$$1 \longrightarrow \pi_1(F) \longrightarrow \pi_1(Y) \longrightarrow \pi_1(S'_o) \longrightarrow 1$$
où $F$ est une fibre générique de $f'$, par \cite[Proposition 12.9]{Cam}. 
Notons $\rho' : \pi_1(Y)\rightarrow \rho(G)$ le morphisme induit par $\rho$. 
Comme l'image de $\pi_1(F)$ dans $\rho(G)$ est triviale,
on en déduit que $\rho'$ factorise par $\pi_1(Y) \rightarrow \pi_1(S'_o)$.

Pour descendre la factorisation à $X$, on va construire un quotient de $S'$ par le 
groupe $\Lambda$ du revêtement $Y/X$. 
Pour cela, on montre que ce groupe agit sur les fibres régulières de $f':Y\rightarrow S'$.
En effet, soit $\Sigma$ une hypersurface connexe de $Y$. Supposons que l'image du groupe fondamental
de $\Sigma$ par $\rho'$ soit triviale, alors $\Sigma$ est inclus dans une fibre de $f'$.
Sinon la restriction de $f'$ à $\Sigma$ serait non constante, donc surjective 
sur la surface de Riemann $S'$ et par factorisation de Stein une fibration sur 
un revêtement étale fini de $S'$. 
En particulier, l'image de $\pi_1(\Sigma)$ serait d'indice
fini dans $\pi_1(S'_o)$ et par la factorisation obtenue précédemment, on obtiendrait que l'image
de $\rho'$ dans $\rho(G)$ serait finie, ce qui est impossible.

Les fibres régulières de $f'$ sont donc les hypersurfaces maximales de $Y$ 
dont le groupe fondamental à une image triviale. On en déduit que le groupe (fini) $\Lambda$
agit sur les fibres régulières de $f'$. Ceci fournit une action sur $S'$.

On note $S$ le quotient par cette action, c'est une orbisurface de Riemann, et on a un
morphisme $f: X \longrightarrow S$.
Nous répétons alors le raisonnement effectué pour $Y$ et $S'$. 
Les groupes fondamentaux des fibres de $f$ ont une image triviale par $\rho$, donc 
en notant $S_o$ la base orbifolde de la fibration $f$, on obtient le résultat attendu,
toujours par \cite[Proposition 12.9]{Cam}.
\end{proof}

\section{Cas des amalgames}
Le théorème \ref{GS} s'applique en particulier aux groupes kählériens qui sont des amalgames. 

En effet (voir \cite[1.4.1 Théorème 6]{Ser}), si 
$G=G_1*_{\Delta}G_2$ alors on peut construire un arbre $T$ semi-homogène sur lequel $G$ agit sans inversion 
avec les propriétés suivantes:
\begin{itemize}
\item il y a deux orbites de sommets sous l'action de $G$, deux sommets adjacents étant dans des orbites 
distinctes;
\item le stabilisateur d'un sommet de la première orbite est isomorphe à $G_1$ et celui d'un sommet de la 
seconde est isomorphe à $G_2$;
\item enfin, le stabilisateur d'une arête est isomorphe à $\Delta$.
\end{itemize}

L'amalgame $G$ est dit non trivial lorsque l'indice de $\Delta$ dans $G_i$ est supérieur ou égal à 2 
pour $i=1$ ou $2$. On ne considére dans la suite que des amalgames non triviaux.
Il est facile de voir que $G$ agit dans ce cas sur l'arbre $T$ sans stabiliser ni bout, ni sous-arbre propre.
De plus, l'arbre $T$ a une infinité de bouts dès que l'amalgame est \emph{strict}, c'est-à-dire si l'indice de 
$\Delta$ dans $G_1$ est supérieur à 2, et l'indice de $\Delta$ dans $G_2$ est supérieur à 3.

Nous souhaitons préciser des conséquences du théorème \ref{GS}, premièrement sur les groupes kählériens simples, 
et deuxièmement sur les noyaux des représentations de groupes kählériens dans des groupes d'automorphismes d'arbres. En application de ces résultats, nous obtiendrons les résultats souhaités sur les groupes de 
Burger-Mozes.

En effet, Burger et Mozes ont montré que les groupes qu'ils ont construits sont des amalgames stricts :
\begin{prop} 
\cite[Theorem 5.5]{BM}
Soit $\Gamma \in \mathcal{F}_{BM}$, alors $\Gamma$ admet des sous-groupes $\Gamma_1$, $\Gamma_2$, $\Delta$, tels que $\Delta$ est d'indice
$2$ dans $\Gamma_1$, d'indice $2n$  dans $\Gamma_2$, où $2<n\in \mathbb{N}$,
et que $\Gamma$ s'écrive sous forme d'amalgame de la 
manière suivante : 
$$\Gamma=\Gamma_1*_{\Delta}\Gamma_2.$$ 
\end{prop}

\section{Groupes kählériens simples et propriété (FA)}

L'existence d'une fibration sur une orbisurface de Riemann, indépendamment de la factorisation, 
impose des restrictions sur les actions possibles d'un groupe kählérien sur un arbre.
Ici, on va montrer qu'un groupe kählérien simple vérifierait la propriété (FA) de Serre, 
et ceci implique que les groupes de Burger et Mozes ne sont pas kählériens.

\begin{defn}
Un groupe $G$ vérifie la \emph{propriété (FA)} si toute action sans inversion de $G$ sur un arbre $T$
fixe au moins un sommet de $T$.
\end{defn}

Serre a caractérisé les groupes ayant la propriété (FA) de la manière suivante.

\begin{thm}
\label{FA}
\cite[Théorème 15]{Ser}
Soit $G$ un groupe, alors $G$ a la propriété (FA) si et seulement s'il vérifie les trois 
propriétés suivantes:
\begin{itemize}
\item ce n'est pas un amalgame non trivial; 
\item il n'a pas de quotient isomorphe à $\mathbb{Z}$;
\item il est de type fini.
\end{itemize}
\end{thm}

Cette caractérisation est déterminante pour notre résultat, puisque le théorème \ref{GS} 
permet d'exclure les possibilités pour un groupe n'ayant pas la propriété (FA).

\begin{prop}
\label{simplek}
Soit $G$ un groupe kählérien. Si $G$ est simple alors $G$ a la propriété (FA).
\end{prop}

\begin{proof}
Soit $X$ une variété kählérienne compacte et $G:=\pi_1(X)$ son groupe fondamental.
Supposons que $G$ est simple. 

\begin{itemize}
\item Puisque $X$ est compacte, son groupe fondamental est de type fini. 

\item Si $G$ a un quotient isomorphe à $\mathbb{Z}$ alors $G$ est 
isomorphe à $\mathbb{Z}$ puisque $G$ est simple. 
Or $\mathbb{Z}$ n'est pas kählérien car son rang est impair.

\item Si $G$ peut s'écrire comme un amalgame non trivial, on distingue deux cas, 
selon que l'amalgame est strict ou non.

\begin{itemize}
\item Si $G$ est un amalgame non strict, alors $G$ a pour 
image dans le groupe d'automorphisme de la chaîne doublement infinie le groupe 
dihédral infini
qui n'est pas simple (il contient $\mathbb{Z}$ comme groupe distingué d'indice deux par exemple).

\item Enfin, si $G$ est un amalgame strict, alors on peut appliquer le théorème \ref{GS}. 
On obtient un morphisme surjectif de $G$ vers le groupe fondamental d'une orbisurface
de Riemann hyperbolique $S$. C'est un isomorphisme car $G$ est simple.

Un groupe d'orbisurface hyperbolique est résiduellement fini 
(on peut appliquer le théorème de Mal'cev puisqu'un groupe d'orbisurface hyperbolique est un 
sous-groupe finiment engendré de $\mathrm{PSL}_2(\mathbb{R})$). 

Or $G$ n'est pas résiduellement fini (car simple infini), 
donc on aboutit à une contradiction.
\end{itemize}
\end{itemize}

On prouve ainsi, par le théorème \ref{FA}, que le groupe $G$ a la propriété (FA).
\end{proof}

\begin{proof}[Démonstration du corollaire \ref{nonK}]
Supposons que $X$ soit une variété kählérienne compacte dont le groupe fondamental $\pi_1(X)$
est isomorphe à $\Gamma$.
En particulier, $\pi_1(X)$ est un groupe kählérien simple.
Par la proposition \ref{simplek}, ce groupe a la propriété (FA).
Or $\Gamma$ est un amalgame strict, et agit sur l'arbre associé sans inversion 
et sans sommet fixé, donc on aboutit à une contradiction.
\end{proof}

Le résultat n'est bien sûr pas vrai dans le cas d'un groupe simple infini non kählérien, et les
groupes de Burger-Mozes sont des contre-exemples.

\section{Restrictions sur les noyaux}

\begin{prop}
\label{mor}
Soient $T$ un arbre infini ayant plus de trois bouts, et $H$ un sous-groupe de $\mathrm{Aut}(T)$
qui ne stabilise ni bout, ni sous-arbre propre de $T$.
Soient $G$ un groupe kählérien et $\rho: G\rightarrow H$ un morphisme surjectif.
Alors 
\begin{itemize}
\item soit le noyau de $\rho$ n'est pas de type fini, 
\item soit $H$ est isomorphe à un groupe d'orbisurface de Riemann hyperbolique.
\end{itemize}
\end{prop}

\begin{proof}
Notons $X$ une variété kählérienne compacte de groupe fondamental isomorphe à $G$.

Appliquons le théorème \ref{GS} à $G$.
On obtient une orbisurface de Riemann hyperbolique $S$, 
une application holomorphe $f:X \rightarrow S$, surjective à fibres connexes,
 et un morphisme 
$\phi : \pi_1(S) \rightarrow H$ tel que $\rho=\phi \circ f_*$.
Notons $K$ le noyau de $\rho$ et $N$ celui de $\phi$.
Notons aussi $\Sigma=\pi_1(S)$.

Puisque $\rho$ est surjectif, on a la suite exacte :
$$1\longrightarrow K \longrightarrow G \stackrel{\rho}{\longrightarrow} H \longrightarrow 1.$$

En appliquant $f_*$ à cette suite exacte, on obtient une suite exacte :
$$1\longrightarrow N \longrightarrow \Sigma \stackrel{\phi}{\longrightarrow} H \longrightarrow 1.$$
En effet, $N$ est l'image par $f_*$ de $K$ car $f_*$ est surjectif, et $\rho=\phi \circ f_*$.
Le groupe $N$ est par conséquent un quotient de $K$. 

Pour montrer la proposition, il suffit de considérer le cas où $K$ est de type fini.
Dans ce cas, comme $N$ est un quotient de $K$, $N$ est aussi de type fini.

Or, un groupe fondamental $\Sigma$ d'une orbisurface de Riemann vérifie la propriété suivante 
(voir \cite{Gri} ou \cite{Grb}):  
les sous-groupes distingués de type fini non triviaux de $\Sigma$ sont d'indice fini.

Ce résultat implique que $N$ est trivial car $H$ est 
nécessairement un groupe infini.
Cela signifie que $H$ est isomorphe à $\Sigma$ puisque $\phi$ est surjective
de noyau trivial.
\end{proof}

Nous avons maintenant tous les outils pour montrer le théorème \ref{noyau} :

\begin{proof}[Démonstration du théorème \ref{noyau}]
Soit $\Gamma$ un groupe de Burger-Mozes, $G$ un groupe kählérien et $\rho : G\rightarrow \Gamma$
un morphisme surjectif. Supposons de plus que le noyau $K$ de $\rho$ soit de type fini.
Alors par la proposition \ref{mor}, $\Gamma$ est isomorphe à un groupe d'orbisurface hyperbolique. 
On a vu qu'un tel groupe est résiduellement fini, ceci implique donc que $\Gamma$ est 
résiduellement fini et on aboutit encore à une contradiction avec la simplicité du 
groupe infini $\Gamma$.
\end{proof}

Par contre, si $H$ est un groupe de type fini quelconque, alors il existe une infinité de groupes 
kählériens qui admettent un morphisme surjectif sur $H$.
En effet, les groupes de surfaces sont des groupes kählériens, qui admettent des quotients libres.
En choisissant une surface de genre suffisament élevé, on peut trouver un morphisme surjectif vers
un groupe libre au nombre de générateur arbitrairement grand. 

\textbf{Remerciements} : Je remercie mon directeur de thèse Philippe Eyssidieux pour m'avoir proposé d'étudier ces questions.

\bibliographystyle{alpha}
\bibliography{biblio}

\end{document}